\documentclass[12pt]{amsart}
 \usepackage{amssymb}
 \usepackage{amsmath}
 \usepackage{lipsum}
\usepackage{tikz-cd}
\usepackage{amsthm}
 \usepackage[utf8]{inputenc}
\usepackage[english]{babel}
 \usepackage{amsmath}
 \usepackage{MnSymbol}
 \usepackage{microtype}
 \usepackage{adjustbox}

\newtheorem{thm}{Theorem}
\newtheorem{cor}[thm]{Corollary}
\newtheorem*{thm-non}{Theorem}
\newtheorem{lem}[thm]{Lemma}
\theoremstyle{definition}
\newtheorem{definition}[thm]{Definition}

\theoremstyle{remark}
\newtheorem{remark}[thm]{Remark}

\theoremstyle{remark}
\newtheorem{Example}[thm]{Example}

\theoremstyle{Proposition}
\newtheorem{prop}[thm]{Proposition}
\newtheorem*{prop-non}{Proposition}

\usepackage{thmtools}
\usepackage{thm-restate}

\usepackage{hyperref}

\usepackage{cleveref}

\DeclareMathOperator{\Spec}{Spec}

\DeclareMathOperator{\Exc}{Exc}

\DeclareMathOperator{\W_2(k)}{W_2(k)}

\DeclareMathOperator{\Supp}{Supp}

\DeclareMathOperator{\Diff}{Diff}
\DeclareMathOperator{\coeff}{coeff}

\DeclareMathOperator{\depth}{depth}
\DeclareMathOperator{\codim}{codim}

\DeclareMathOperator{\sO}{\mathcal{O}}

\numberwithin{equation}{section}

\begin{document}

\newcommand{\nospaceperiod}{\makebox[0pt][l]{\,.}}
\newcommand{\nospacepunct}[1]{\makebox[0pt][l]{\,#1}}

\newcommand*{\isoarrow}[1]{\arrow[#1,"\rotatebox{90}{\(\sim\)}"
]}

\newcommand\blankpage{%
    \null
    \thispagestyle{empty}%
    \addtocounter{page}{-1}%
    \newpage}

\newcommand{\dotr}[1]{#1^{\bullet}} 
\newcommand{\defeq}{\mathrel{\mathop:}=}
\newcommand{\eqdef}{\mathrel{\mathop=}:}
\newcommand{\C}{\mathbb{C}}
\newcommand{\Q}{\mathbb{Q}}
\newcommand{\Z}{\mathbb{Z}}
\newcommand{\R}{\mathbb{R}}
\newcommand{\mO}{\mathcal{O}}
\newcommand{\mL}{\mathcal{L}}

\newcommand{\EA}[1]{\textcolor{blue}{(EA: #1)}}

\bibliographystyle{alpha}

\setcounter{tocdepth}{1}

\title{On a vanishing theorem for birational morphisms of threefolds in positive and mixed characteristics}
\title[A birational vanishing theorem in positive characteristics]{On a vanishing theorem for birational morphisms of threefolds in positive and mixed characteristics}
\author{Emelie Arvidsson}
\address{University of Utah, Department of Mathematics, JWB 311, 84112 Salt Lake City, Utah, USA}
\email{arvidsson@math.utah.edu}

\thanks{The author was supported by SNF \#P500PT\_203083}




\begin{abstract} In this short note, we prove a relative Kawamata--Viehweg vanishing type theorem for birational morphisms. We use this to prove a Grauert–Riemenschneider theorem over 
log canonical threefolds without zero-dimensional lc-centers, in residue characteristic $p>5$. 
In large enough residue characteristics, we prove a Grauert–Riemenschneider theorem over threefold log canonical singularities with standard coefficients. These vanishing theorems can also be used to study the depth of log canonical singularities at non log canonical centers, as well as the singularities of the log canonical centers themselves. 
The former simplifies joint work with F. Bernasconi and Zs. Patakfalvi, and the latter appears in joint work with Q. Posva.

\end{abstract}

\maketitle

\section{Introduction} 
Over the complex numbers, Kawamata log terminal singularities are Cohen--Macaulay and rational \cite{Elk81}. It is therefore natural to ask if Kawamata log terminal singularities also have such good properties over algebraically closed fields of positive characteristics? S. Kov\'acs showed that these two questions are essentially equivalent to the single question; if klt singularities are Cohen--Macaulay or not over fields of positive characteristics? In particular, he showed that Cohen--Macaulay klt singularities are pseudo--rational, and when resolution exists, rational in the usual sense \cite{kovacs2017rational}. In dimension three, it was recently showed that klt singularities are indeed Cohen--Macaulay in characteristic $p>5$ \cite{ABL, HW}, and counter--examples where constructed in low characteristics \cite{CT, Berfail, ABL}.\\

With the positive results for threefold klt singularities at hand, it is natural to ask if Koll\'ar's theorem on the depth of log canonical singularities at non log canonical centers \cite[Theorem 3]{localkvv} also hold true for threefolds in (possibly large enough) positive characteristics? Let us first recall what Koll\'ar's theorem says over the Complex numbers.

\begin{thm}[Koll\'ar]\label{lkvv} Let $(X,\Delta)$ be a three dimensional log canonical pair over $\mathbb{C}$ and $x\in X$ a point which is not a log-canonical center then 
$$\depth_x(\mathcal{O}_X)=\codim_X(x)$$

\end{thm}

 It turns out that Theorem \autoref{lkvv} fails to be true in every positive characteristic (see Example \autoref{Ex}). In this note we use the previous results on dlt threefold singularities in characteristic $p>5$ to study the depth of log canonical singularities at non lc-centers via a dlt modification $f\colon Y\to X$. Our main result is that $f$ can be chosen in such a way that it satisfies a birational Kawamata--Viehweg type vanishing theorem (see Prop. \autoref{dltmodif}). As a direct consequence of this birational vanishing theorem, we obtain the following, which can be considered as a positive characteristics analog of Theorem \autoref{lkvv}.
 
\begin{thm}[Corollary \autoref{wild}]\label{introwild} Assume that $X$ is an excellent scheme whose closed points have perfect residue field of characteristic $p>5$. Let $x\in (X,\Delta)$ be a log--canonical threefold singularity which is not a log--canonical center. Suppose $x\in C$, where $C$ is a curve and a minimal lc--center of  $ (X,\Delta)$. Then there exists a crepant dlt modification $f\colon (Y,\Delta_Y)\to (X,\Delta_X)$, such that every component of $E:=\Exc(f)$ has coefficient one in the crepant boundary $\Delta_Y$, and maps surjectively onto $C$. Moreover, $H^2_x(X,\mathcal{O}_X)\cong H^0_x(C,R^1{f_{|E}}_*\mathcal{O}_E)$, where $f_{|E}$ denotes the restriction of $f$ to the surface $E$.
\end{thm} 

If $x\in X$ is not contained in a one dimensional minimal log canonical center of $(X,\Delta)$ then $x\in (X, \Delta)$ is either klt or plt and $H^2_x(X,\mathcal{O}_X)=0$, \cite{ABL, BK}. The above theorem therefore says that the failure of the theorem of Koll\'ar (Theorem \autoref{lkvv}) in positive characteristic is completely explained by the possible failure of the higher direct images of certain surface fibrations to be torsion-free in positive characteristics, see remark \autoref{wildrem}. A similar result is also obtained in \cite[Theorem 1.5]{ABP}, however for a different modification $f$ and a lc pair $(Y, \Delta_Y)$ that  in general is not dlt. Both the statement and proof of Theorem \autoref{introwild} differ from the approach taken in \cite[Theorem 1.5]{ABP}, and Theorem \autoref{introwild} can be used to simplify several arguments in that paper, see Corollary \autoref{S2}.\\

The birational vanishing theorems established in this note also implies that we can extend the \emph{weak Grauert–Riemenschneider} theorem \cite[Theorem 3]{BK} to log canonical threefold pairs in certain situations. First, we have the following over log canonical pairs without zero-dimensional lc-centers:

\begin{thm}[Corollary \autoref{GR}] Let $(X,\Delta)$ be a log canonical threefold singularity without zero dimensional lc-centers, where $X$ is an excellent scheme whose closed points have perfect residue field of characteristic $p>5$. Then the weak Grauert–Riemenschneider theorem holds over $(X,\Delta)$. I.e., for every log resolution $\mu\colon (V, \Delta_V)\to 
(X,\Delta)$ and every $\mathbb{Z}$-divisor $D'$ on $V$, such that there exists a $\mathbb{Q}$-divisor $\Delta'$ with $g_*\Delta'\leq \Delta$, $\lfloor \Exc(\Delta')\rfloor=0$, and satisfying $D'\sim_{\mathbb{R},\mu} K_{X'}+\Delta'$, we have  $R^i\mu_*(D')=0$, for all $i>0$. 
\end{thm}

Secondly, by invoking a recent vanishing theorem for log Calabi-Yau surfaces due to Kawakami \cite{KKVV}, we can show a Grauert-Riemenschneider vanishing theorem over more general log-canonical singularities, in large enough characteristic.

\begin{thm}[Corollary \autoref{GRII}] There exists an integer $p_0$ such that if $(X,\Delta)$ is a log canonical threefold with standard coefficients over a perfect field of characteristic $p>p_0$, then the weak GR-vanishing theorem holds over $(X,\Delta)$. I.e., for every log resolution $\mu\colon (V, \Delta_V)\to 
(X,\Delta)$ and every $\mathbb{Z}$-divisor $D'$ on $V$, such that there exists a $\mathbb{Q}$-divisor $\Delta'$ with $g_*\Delta'\leq \Delta$, $\lfloor \Exc(\Delta')\rfloor=0$, and satisfying $D'\sim_{\mathbb{R},\mu} K_{X'}+\Delta'$, we have  $R^i\mu_*(D')=0$, for all $i>0$. 
\end{thm}

Finally, let us mention an application, joint with Q. Posva, to the study of log canonical centers. In \cite{AP}, we discuss some (semi-) normality properties of lc-centers in (sufficiently large) positive characteristics, see Theorem \autoref{AP} and Theorem \autoref{AP2} . The proofs of these theorems uses the birational vanishing obtained in Propositions \autoref{dltmodif} and \autoref{dltmodifII}, respectively.

\section{Acknowledgments} The original motivation for this note was to give an alternative formulation, and a simpler proof, of the main technical theorem in \cite{ABP}, (Theorem 1.5). We are therefore indebted to our coauthors in that project, Fabio Bernasconi and Zsolt Patakfalvi, for discussions related to that work. We would like to express our gratitude to Christopher Hacon for useful discussions and for reading previous drafts of this note. In particular, we thank him for several useful suggestions and corrections, and we are indebted to him for suggesting Proposition \autoref{dlt with ample}.  We also like to thank Quentin Posva for useful discussions. The author was supported by SNF \#P500PT\_203083.
\section{preliminaries} 

We will be working over an excellent local domain $R$ of finite Krull dimension and admitting a dualizing complex. In this article the residue field of $R$ will be assumed to be perfect of positive characteristic. \\

By a variety we shall mean a separated reduced, equidimensional, excellent scheme over $R$.  \\

The dimension of a variety $X$  shall always mean the absolute dimension and not the relative dimension of $X$ over $R$.\\


A $\mathbb{Q}$-divisor $D$ will be said to be $\Q$-Cartier if there exists an integer $m$ such that $mD$ is a Cartier divisor. A $\mathbb{Q}$-divisor will be said to be ample/nef/big if it is $\Q$-Cartier and an integer multiple is ample/nef/big as a line bundle. 
\\

A point $x\in X$ will be said to be $\mathbb{Q}$-factorial if for every $\mathbb{Z}$-divisor $D$ on $X$, there exists some neighborhood $U$ of $x$ such that $D_{|U}$ is $\mathbb{Q}$-Cartier. $X$ is said to be $\mathbb{Q}$-factorial if it is $\mathbb{Q}$-factorial at every point. \\

By a log pair $(X, \Delta)$, we shall mean a normal variety $X$ together with an effective $\Q$-divisor $\Delta$ on $X$ such that $K_X+\Delta$ is $\Q$-Cartier. By a klt (respectively lc, dlt) pair we shall mean a log pair with klt (respectively lc respectively dlt) singularities in the sense of \cite[Section 2]{KolMor} and \cite[Section 2.5]{BPTWW}. By \cite[Corollary 4.11]{Tan16} if $X$ is a surface and $(X, \Delta)$ is a dlt pair then $X$ is $\Q$-factorial.\\

We say that $\Delta$ has standard coefficients if they are contained in the set $\{1-\frac{1}{m}|m\in\mathbb{Z}\}\bigcup 1$.\\

A normal surface $X$ will be said to be a \emph{del Pezzo surface} if it has klt singularities and $-K_X$ is a $\Q$-Cartier and ample divisor. A log pair $(X, B)$ will be said to be a \emph{log del Pezzo surface} if it has klt singularities and $-(K_X+ B)$ is ample. \\

A normal surface $X$ will be said to be a \emph{Calabi-Yau surface} if it has lc singularities and $K_X$ is $\Q$- Cartier and numerically trivial. A log pair $(X, B)$ will be said to be a \emph{log Calabi-Yau surface} if it has lc singularities and $(K_X+ B)$ is numerically trivial. \\

For a $\mathbb{Q}$-divisor $D$, we will denote by $\lfloor D\rfloor $ and $\lceil D\rceil $ the round-down and round-up of $D$, respectively. The set of coefficients of $D$ will be denoted $\coeff(D)$, and for a irreducible divisor $E$ we will denote by $\coeff_E(D)$ the coefficient of $E$ in $D$.\\

For a log canonical pair $(X,E+ \Delta)$, where $E$ is an irreducible divisor disjoint from the support of $\Delta$ and with normalization $\nu\colon E^{\nu}\to E$, we will denote by 
$\Diff_{E^{\nu}}(\Delta)$ the different as definied in \cite[Defn. 4.2]{Kolsing}\\

If $g\colon Y\to X$ is a birational morphism between normal varieties then $\Exc(g)$ will denote the reduced exceptional divisor of $g$. If $\Delta$ is a divisor on $Y$ we will denote $\Exc(\Delta)$ the exceptional part of $\Delta$.  \\

\begin{definition}We will say that a proper morphism $g\colon Y\to X$ between normal varieties satisfies Kawamata--Viehweg vanishing, abbreviated \emph{KVV}, if whenever \begin{itemize} \item $D$ is a $\mathbb{Q}$-Cartier $\Z$-divisors on $Y$, and \item $\Delta$ is a $\mathbb{Q}$-boundary such that $(Y,\Delta)$ is klt and satisfying \item $D-(K_Y
+\Delta)$ is $g$-ample,  \end{itemize}  then $R^ig_*\sO_Y(D)=0$ for all $i>0$. \end{definition}

\begin{remark}\label{dltandnef}
Note that if $(Y,\Delta)$ is dlt and $D-(K_Y+\Delta)\sim_{\mathbb{Q},g}A$ is $g$-ample, then for $\epsilon$ small enough $\Delta'=(1-\epsilon)\Delta$ satisfies that $(Y,\Delta')$ is klt and that $D-(K_Y+\Delta')\sim_{\mathbb{Q},g}A+\epsilon \Delta$ is $g$-ample, \cite[Prop. 2.4.2]{KolMor}. Similarly, if $(Y,\Delta)$ is klt and $D-(K_Y+\Delta)$ is $g$-nef, and if there exists an effective $g$-anti-ample divisor $A$ on $Y$. Then, for all small enough $\epsilon$, the pair $(Y,\Delta+\epsilon A)$ is klt and $D-(K_Y+\Delta+\epsilon A)$ is $g$-ample. \end{remark}

\begin{remark}\label{assumfnef} If $(Y,\Delta)$ is dlt and $D-(K_Y+\Delta)$ is $g$-nef then the KVV type assumptions are usually not met, even if $Y$ admits an effective anti-$g$-ample divisor $A$. In this situation, we often additionally assume that we can choose $A$ to be $g$-exceptional and that $\lfloor \Delta \cap \Exc(g) \rfloor=0$ (see e.g., Corollary \autoref{GR}). Under this extra assumption, indeed $(Y,\Delta+\epsilon A)$ is dlt and $D-(K_Y+\Delta+\epsilon A)$ is $g$-ample for all $\epsilon>0$ small enough. \end{remark}

\section{A birational Kawamata--Viehweg vanishing theorem}

The following is essentially \cite[Prop. 3.6]{cascini2}.

\begin{prop}\label{birKVV}Let $f\colon X\to Y$ be a proper birational morphism of normal varieties such that there exists an $f$-ample $\mathbb{Z}$-divisor $-E$ satisfying: \begin{itemize}

\item $E>0$ and $\Supp(E)=\Exc(f)$ and every component $E_i$ of $E$ is $\Q$-Cartier.
\item Every reflexive rank-one sheaf on $X$ is $S_3$.
\item For every component $E_i$ of $E$ let $ \nu_i \colon E^{\nu}_i\to E_i$ denote the normalisation of $E_i$ and for $C_i=f(E_i)$ let $\nu_{C_i}\colon C^{\nu}_i\to C_i$ denote the normalization of $C_i$. Suppose that for all $i$ the morphism $g_i$  in the diagram below
$$\begin{tikzcd}
E^{\nu}_i\arrow{r}{\nu_i}\arrow{d}{g_i}& E_i\arrow{d}{f_{|E_i}}\\
C^{\nu}_i\arrow{r}{\nu_{C_i}}& C_i\end{tikzcd}$$
satisfies KVV.
\end{itemize}
For $\Delta\geq 0$ write $\Delta=B+\Delta_0$, where $\Delta_0$ is $f$-exceptional and $B$ has no $f$-exceptional component. 
Then, if $D$ is a $\Q$-Cartier $\mathbb{Z}$-divisor such that $D-(K_X+\Delta)$ is $f$-ample for a boundary $\Delta$, with $\coeff(\Delta)\leq 1$ and such that $(X, B+\Supp(E))$ is dlt, then $R^qf_*\sO_X(D)=0$ for all $q>0$. 
\end{prop}

\begin{proof} 
Define $D=:D_0$, $\lambda_0:=0$. Let $A\defeq D-(K_X+\Delta)=D-(K_X+B+\Delta_0)$.
We will inductively define triples $\{(\lambda_j, D_j,\Delta_j)\}_{j\in \mathbb{Z}}$ such that 

\begin{enumerate}
\item $\lambda_j$ is a rational number such that 
$$0\leq \lambda_0 \leq \lambda_1 \leq \dots  \leq \lambda_j.$$
\item $\Delta_j$ is an effective $f$-exceptional divisor with $\coeff(\Delta_j)\leq 1$.
\item $D_j$ is a $\mathbb{Z}$-divisor such that $D_j-(K_X+B+\Delta_j)\sim_{\mathbb{Q}}A-\lambda_jE$.
\item If $R^qf_*\sO_X(D_j)=0$ for all $q>0$ then $R^qf_*\sO_X(D_{j-1})=0$ for all $q>0$.
\item $\lim_{j\to \infty}\lambda_j=\infty$.
\end{enumerate}

Let $\mu_j$ be the non-negative rational number such that $\coeff(\Delta_j+\mu_jE)\leq1$ and such that there exists at least one component $E_{i_j}$ of $\Delta_j+\mu_jE$ that is of coefficient one. Note that we allow $\Delta_j=0$ and in that case $\mu_j=\frac{1}{a}$ where $a$ is a maximal coefficient of $E$. It could also happen that $\Delta_j$ already has at least one coefficient which is equal to one and it that case the definition gives $\mu_j=0$. \\

We define $\lambda_{j+1}:=\lambda_j+\mu_j$, $\Delta_{j+1}=\Delta_j+\mu_jE-E_{i_j}$,  $D_{j+1}:=D_j-E_{i_j}$. By the assumption that every component of $E$ is $\mathbb{Q}$-Cartier, $K_X+B+ \Delta_j$ and $D_j$ are $\mathbb{Q}$-Cartier for all $j$ and $(1)_j$ and $(2)_j$ hold. We have 

$$D_{j+1}-(K_X+B+\Delta_{j+1})=D_j-(K_X+B+\Delta_{j})-\mu_jE$$
By induction $$D_j-(K_X+B+\Delta_{j})=A-\lambda_jE,$$ from this $(3)_j$  follows since $\lambda_{j+1}=\lambda_j+\mu_j$.\\

We will show $(4)_j$. To this end, since $\Supp(B+\Delta_{j+1})$ does not contain $E_{i_j}$, we have that $(X, B+\Delta_{j+1}+E_{i_j})$ is dlt, and 
$$D_j-(K_X+B+E_{i_j} +\Delta_{j+1})=A-\lambda_{j+1}E$$
is $f$-ample. 
Since $\sO_X(D_{j+1})$ and $\sO_X(D_j)$ are $S_3$ there exists a divisor $D_{i_j}$ on $E^{\nu}_{i_j}$ and an exact sequence (see \cite{HW} and \cite[Lemma 5]{BK})

$$\begin{tikzcd} 0\arrow{r} &\sO_X(D_{j+1})\arrow{r} &\sO_X(D_j)\arrow{r}&{\nu_{i_j}}_*\sO_{{E_{i_j}^{\nu}}}(D_{i_j})\arrow{r}&0\end{tikzcd}$$

such that  $D_{i_j}\sim_{\mathbb{Q}}{D_j}_{|E^{\nu}_{i_j}}-\Sigma$ where $\Sigma\leq \Diff_{E^{\nu}_{i_j}}(B+\Delta_j)$.  By adjunction \cite[Definition 4.2]{KolMor}, $$(K_X+B+E_{i_j} +\Delta_{j+1})_{|E^{\nu}_{i_j}}\sim_{\mathbb{Q}} K_{E^{\nu}_{i_j}}+\Diff_{E^{\nu}_{i_j}}(B+\Delta_{j+1}).$$

Therefore,  $D_{i_j} -(K_{E^{\nu}_{i_j}}+\Diff_{E^{\nu}_{i_j}}(B+\Delta_{j+1})-\Sigma)$ is $g_{i_j}$-ample. Since $0\leq \Diff_{E^{\nu}_{i_j}}(B+\Delta_{j+1})-\Sigma \leq \Diff_{E^{\nu}_{i_j}}(B+\Delta_{j+1})$ it follows from KVV for $g_{i_j}$ that $R^q{g_{i_j}}_*\sO_{{E_{i_j}^{\nu}}}(D_{i_j})=0$ for all $q>0$, see Remark \autoref{dltandnef}.
Since $\nu_{C_{i_j}}$ is affine this implies that $R^q(\nu_{C_{i_j}}\circ g_{i_j})_*=0$ for all $q>0$. By the commutative diagram and the fact that $\nu_{i_j}$ is affine it hence follows that $0=R^q(f_{|E_{i_j}}\circ \nu_{i_j})_*=R^q{f_{|E_{i_j}}}_*\circ {\nu_{i_j}}_*$. 
Therefore, by the long exact sequence of cohomology of right derived functors, we conclude that if $R^qf_*\sO_X(D_j)=0$ then $R^qf_*\sO_X(D_{j-1})=0$. This is $(4)$. \\

Finally, we show $(5)$. There exists an infinite increasing subsequence of positive integers $\{j_k\}_{k\in \mathbb{Z}_{>0}}$ such that $i_{j_k}$ is constant for all $k$. By possibly reordering the components of $E$ we may assume that $i_{j_k}=1$ for all $k$, i.e., $E_{i_{j_k}}=E_1$ for all $k\in \mathbb{Z}_{>0}$. Let $\{{j_k}\}$ be a maximal such sequence, that is, for any positive integer $r$ such that ${i_{j_k}}+r<i_{j_{k+1}}$, then $E_{{i_{j_k}}+r}\neq E_1$. By construction, for any $k\in \mathbb{Z}_{>0}$, we find $\coeff_{E_1}(\Delta_{j_k})=1$ and $\coeff_{E_1}(\Delta_{j_k+1})=0$. Denote by $e_1:=\coeff_{E_1}(E)$. Then, $$1=\coeff_{E_1}(\Delta_{j_{k+1}})=e_1\mu_{j_{k+1}}+\coeff_{E_1}(\Delta_{j_{k+1}-1}).$$
Since $j_{k+1}-1\neq j_k$ we find  $\coeff_{E_1}(\Delta_{j_{k+1}-1})=e_1\mu_{j_{k+1}-1}+\coeff_{E_1}(\Delta_{j_{k+1}-2})$.
For ever $k\in \mathbb{Z}_{>0}$ let $r_k\in \mathbb{Z}$ be such that $j_k+r_k=j_{k+1}$. Inductively, and using that $\coeff_{E_1}(\Delta_{{j_k}+1})=0$ we find 
$$1=\coeff_{E_1}(\Delta_{j_{k+1}})=\sum^{r_k}_{i=1}e_1\mu_{j_k+i}.$$ In particular, $$\sum_{k=1}^{\infty}\sum^{r_k}_{i=1}\mu_{j_k+i}=\infty,$$
since $\lambda_j=\sum_{i=0}^{j-1}\mu_i$ it follows that $\lim_{j\to \infty}\lambda_j=\infty$.\\

To prove the theorem it is sufficient to show that $R^qf_*\sO_X(D_j)=0$ for $q>0$ and $j>>0$. By $(5)_j$ and the definition of $D_j$ there exists some anti-ample divisor $E'$ supported on $\Exc(f)$ and such that for every integer $n$ there exists some integer $k_n$ such that $-nE'$ is contained in $D_{j}$ for all $j>k_n$.  Let $n>>0$ be such that $R^qf_*(-jE'+N)=0$ for all $j\geq n$ and all $f$- nef divisors $N$ \cite[Theorem 1.5]{Keeler}. By $(3)_j$ and $(5)_j$ we have that $D_j$ is $f$- ample for $j>>0$ and we can write $D_{j}\sim_{f,\mathbb{Q}} A'-\lambda'_jE$ for some $f$-ample $A'$ and some $\lambda'_j>0$ such that $\lim_{j\to \infty}\lambda'_j=\infty$. It follows from this, that for $j>>0$ big enough $D_{j}+nE$ is $f$-nef.  Therefore, it follows that $R^qf_*\sO_X(D_j)=0$ for $q>0$ and $j>> k_n$.


\end{proof} 

\section{Existence of dlt models with exceptional ample divisors}

In order to apply Proposition \autoref{birKVV} to a dlt modification $f$ of a log canonical singularity we need to assure the existence of an anti--effective, $f$-ample and exceptional divisor. If we reduce our attention to $\mathbb{Q}$-factorial schemes this can easily be assured.

\begin{lem}Let $f\colon Y \to X$ be a proper morphism between normal quasi--projective varieties. If $X$ is $\mathbb{Q}$-factorial then there exists an anti--effective, $f$-ample and exceptional divisor $A$ on $Y$ such that $\Supp(A)=\Exc(f)$.
\end{lem}

\begin{proof}Let $A$ be an ample divisor on $X$. Then $A'=f_*A$ is $\mathbb{Q}$-Cartier and $A-f^*(A')$ is exceptional, anti--effective and $f$-ample. 
\end{proof} 

Recent work of Koll\'ar, \cite{KolMMP}, shows that we can find a dlt--modification in dimension three which contains an ample exceptional divisor, also without assuming that the pair $(X,\Delta_X)$ is $\mathbb{Q}$-factorial. This dlt-modification will however, in general, not $\mathbb{Q}$-factorial.

\begin{prop}\label{dlt with ample} Let $(X,\Delta_X)$ be a log canonical threefold. Then there exists a proper crepant birational morphism $f\colon (Y,\Delta_Y)\to (X,\Delta_X)$ such that every component of $\Exc(f)$ is $\mathbb{Q}$-Cartier and appear with coefficient one in $\Delta_Y$, $(Y, \Delta_Y)$ is dlt, and there exists an ample $\mathbb{Q}$-Cartier $\mathbb{Z}$-divisor $A$ supported on $\Exc(f)$.
\end{prop}

\begin{proof}Let $g\colon (V,\Delta_V+\Exc(g))\to (X,\Delta_X)$ be a log resolution of  $(X,\Delta_X)$, where $\Delta_V=g_*^{-1}\Delta_X$ and such that there exists an exceptional $g$-ample anti-effective divisor on $V$. Indeed, such a resolution exist by \cite[Theorem 2.13] {BPTWW} and \cite[Theorem 1]{KWresamp}.  Let $\Theta=\Delta_V+\Exc(g)$, then $K_V+\Theta\sim_{g,\mathbb{Q}}\sum_i(1-a_i)E_i$ for some rational numbers $a_i\leq 1$. I.e.,  $g^*(K_X+\Delta_X)=K_V+\Delta_V+\sum_ia_iE_i$ where $a_i=-a(E_i,X, \Delta)$, for $a(E_i,X, \Delta)$ the discrepancy \cite[Defn. 2.25]{KolMor}. Let $e_i=1-a_i$, there exists real numbers $h_i$ that are linearly independent over $\mathbb{Q}[e_1, \dots ,e_n]$ such that $H=\sum_ih_iE_i$, satisfies that $-H$ is effective and $g$-ample and such that $K_V+\Theta +rH\sim_{g,\mathbb{Q}}\sum_i(e_i+rh_i)E_i$ is $g$-ample for some $r>0$. By \cite[Theorem G]{BPTWW} we can run a $K_V+\Theta$-MMP with scaling of $H$ over $X$, and we do so in the sense of \cite[Theorem 9]{KolMMP}. In particular, a single step of this MMP may contract an extremal face which is not an extremal ray and the steps of this MMP might not be $\mathbb{Q}$-factorial \cite[Warning 1.8.]{KolMMP}.\\

 Suppose this MMP terminates after $m$ steps, with $g_m\colon Y_m\to X$. Denote by $\Theta^m$, $H^m$ and ${E_i}^m$ the strict transforms of $\Theta$, $H$ and each of the ${E_i}'s$ respectively. By \cite[Theorem 2]{KolMMP} the divisor $-\sum_i(1-a_i)E^m_i$ is effective, and $\sum_i(1-a_i)E^m_i+ (r_{m-1}-\epsilon)H^m$ is a $g_m$-ample and $g_m$-exceptional $\mathbb{R}$-divisor for some $r_{m-1}<r$, and all $1>>\epsilon>0$. Moreover, for all $i$, $E_i^m$ is $\mathbb{Q}$-Cartier. The existence of a $g_m$-ample and $g_m$-exceptional $\mathbb{R}$-divisor assures, by possibly perturbing the coefficients, the existence of  a $g_m$-ample and $g_m$-exceptional $\mathbb{Q}$-Cartier $\mathbb{Q}$-divisor. By taking a multiple, we can assure the existence of an ample $\mathbb{Z}$-divisor $A$ supported on $\Exc(g_m)$.  Since the divisor $-\sum_i(1-a_i)E^m_i$ is effective, and by construction, so is $\sum_i(1-a_i)E^m_i$, it follows that  $-\sum_i(1-a_i)E^m_i=0$. Therefore, if $a_i\neq 1$ then $E_i$ was contracted by this MMP. Conversely, If $C$ is any curve not contained in $\Supp(\sum_i(1-a_i)E_i)$, then its intersection with this divisor is non negative and hence no such curve could have been contracted by this MMP. Hence, if $a_i=1$, then $E_i$ was not contracted. Therefore, every exceptional divisor of $g_m$ appears with coefficient $1$ in $\Theta^m$ and $K_{Y^m}+\Theta^m=g_m^*(K_X+\Delta)$, hence $g_m\colon (Y^m, \Theta^m)\to (X, \Delta)$ is crepant.\\

 To conclude the proof of the proposition, it is sufficient to show that $(Y^m, \Theta^m)$ is dlt. To this end, we check that the dlt condition is preserved for every step of this MMP, where the steps are as in \cite[Theorem 9]{KolMMP}. First, we check that being dlt is preserved under any divisorial contraction. Let $\Phi^i\colon (Y^i, \Theta^i)\to (Y^{i+1}, \Theta^{i+1})$ be a divisorial contraction. By \cite[Lemma 3.38, Cor. 3.44]{KolMor} it is sufficient to show that $K_{Y^i}+\Theta^i$ is $\Phi^i$-anti-ample. If $i=0$, and hence $Y^0=V$, we note that we choose $H$ to be ample over $X$, and hence since $\Phi^0$ contracts curves which are $K_{Y^0}+ \Theta^0+r_0H$ trivial, where $r_0>0$, it follows directly that this step is $K_{Y^0}+ \Theta^0$-negative. Let $i>0$, by \cite[Theorem 2, (3)]{KolMMP} we have that $K_{Y^i}+ \Theta^i+(r_{i-1}-\epsilon)H^i$ is $g_i$-ample for all $0<\epsilon <<1$, where $g_i\colon Y^i\to X$ is the structure morphism over $X$. Also, the curves contracted by $\Phi^i$ are $K_{Y^i}+ \Theta^i+r_{i}H^i$-trivial for $r_i<r_{i-1}$.  But these two things together implies that $H^i$ is $\Phi^i$-ample, which then implies that $K_{Y^i}+ \Theta^i$ is $\Phi^i$ anti-ample. This shows that being dlt is preserved by divisorial contractions.\\
 
Lastly, we check that if $(Y^i, \Theta^i)$ is dlt and 
$$\begin{tikzcd} (Y^i, \Theta^i)\arrow{rd}[swap]{\Phi} &&(Y^{i+1}, \Theta^{i+1})\arrow{ld}{\Phi^+}\\
&Z\end{tikzcd}$$ is a flip of the kind appearing in \cite{KolMMP} then $(Y^{i+1}, \Theta^{i+1})$ is dlt. By the same argument as in the divisorial case, we find that $H^i$ is $\Phi$-ample, and therefore $(K_{Y^i}+ \Theta^i)$ is $\Phi$-anti-ample.  By \cite[Theorem 2, (3)]{KolMMP} we have that $K_{Y^{i+1}}+ \Theta^{i+1}+(r_{i}-\epsilon)H^{i+1}$ is $g_{i+1}$-ample for $0<\epsilon <<1$, and since  $K_{Y^{i+1}}+ \Theta^{i+1}+r_{i}H^{i+1}$ is $\Phi^{+}$-trivial. This in turn implies that $H^{i+1}$ is $\Phi^+$-anti-ample, and hence  $K_{Y^{i+1}}+ \Theta^{i+1}$ is $\Phi^{+}$-ample. By \cite[Proof of Lemma 3.38]{KolMor}, we therefore find that for an arbitrary divisor $E$ over $Z$, we have $a(E, Y^i, \Theta^i)\leq a(E, Y^{i+1}, \Theta^{i+1})$ and the inequality is strict for any divisor $E$ with center contained in the flipping or flipped locus. This shows that $(Y^{i+1}, \Theta^{i+1})$ is dlt by \cite[Proof of Cor. 3.44]{KolMor}.


\end{proof}

\section{Vanishing theorems for morphisms of threefolds} 



\begin{lem}\label{point} There exists an integer $p_0$ with the following property. Let $\Delta$ be an effective divisor with standard coefficients.  Suppose that $x\in (X,\Delta)$ is a closed point which is a log canonical center of a three dimensional lc pair with perfect residue field $k_x$ of residue characteristic $p>p_0$ at $x$. Let $f\colon (Y,\Delta_Y+E)\to (X,\Delta_X)$ be a crepant birational contraction of threefolds such that $E$ is a reduced divisor, with support equal to the exceptional locus of $f$ and $(Y, \Delta_Y+ E)$ is dlt. If $E_i$ is a component of $E$ with $f(E_i)=x$, then $(E^{\nu}_i, \Diff_{E^{\nu}_i}(\Delta_Y+E-E_i))$ is a log canonical log Calabi-Yau surface over $k_x$. Moreover,  the following Kawamata--Viehweg vanishing type theorem holds on $E^{\nu}_i$;\\
for any $\mathbb{Z}$-divisor $D$ on $E^{\nu}_i$, such that there exists an effective $\mathbb{Q}$-divisor $\Delta'\leq \Diff_{E^{\nu}_i}(\Delta_Y+E-E_i)$ with $D-K_{E^{\nu}_i}-\Delta'$ ample, we have that $H^q(E^{\nu}_i, \mathcal{O}_{E^{\nu}_i}(D))=0$, for all $q>0$. \end{lem}

\begin{proof}Write $K_Y+\Delta_Y+E\sim_{f, \mathbb{Q}}0$, and hence by adjunction $(E^{\nu}_i,  \Diff_{E^{\nu}_i}(\Delta_Y+E-E_i))$ is log canonical and $K_{E^{\nu}_i}+\Diff_{E_i}(\Delta_Y+E-E_i)\sim_{\mathbb{Q},g_i}0$. Since $\Delta_Y$ is the strict transform of $\Delta$ the pair $(Y,\Delta_Y+E-E_i)$ has standard coefficients. Let $p_0>5 $ then \cite[Theorem 3.1]{BK} implies that $E_i$ is normal and $(E_i,  \Diff_{E_i}(\Delta_Y+E-E_i))$ is therefore log canonical with standard coefficients \cite[4.1, 4.4]{Kolsing}. Since $D-K_{E^{\nu}_i}-\Delta'$ is ample we can perturb $\Delta'$ into a boundary with coefficients strictly less than $1$. The vanishing theorem hence follows from \cite[Theorem 1.1]{KKVV}.
\end{proof}

\begin{lem}\label{curve} Let $f\colon Y\to X$ be a birational contraction of threefolds with $E:=\Exc(f)$ and such that $(Y,E)$ is potentially dlt. If $f(E_i)$ is a curve inside $X$ then the relative Kawamata--Viehweg vanishing theorem holds for $E^{\nu}_i\overset{g_i}{ \rightarrow} f(E_i)^{\nu}$.
\end{lem} 

\begin{proof}By adjunction $E^{\nu}_i$ is a (potentially) klt surface, and therefore $\mathbb{Q}$-factorial \cite[Corollary 4.11]{Tan16}, thus klt. The morphism $g\colon E^{\nu}_i\to f(E_i)^{\nu}$ induced by restriction is a contraction onto a normal curve. Therefore, \cite[Theorem 3.3]{Tan16} implies that $R^qg_*\sO_{E^{\nu}_i}(D)=0$ for all $q>0$, and for every $\mathbb{Q}$-Cartier divisor $D$ that satisfies $D-(K_{E^{\nu}_i}-\Delta)$ is $g$-nef and big for some boundary $\Delta$, such that $( E^{\nu}_i, \Delta)$ is klt. 
\end{proof}

\begin{prop}[Birational Kawamata--Viehweg vanishing]\label{KVVmain} Let $f\colon X\to Y$ be a projective morphism of normal threefolds with closed points having perfect residue fields of characteristic $p>5$. Assume that there exists an $f$-ample $\mathbb{Z}$-divisor $-E$ with $\Supp(E)=\Exc(f)$ and such that $(X,\Supp(E))$ is potentially dlt. For any divisor $\Delta \geq 0$ write $\Delta=B+\Delta_0$ where $\Delta_0$ is $f$-exceptional and $B$ is without $f$-exceptional component. Let $D$ be a $\Q$-Cartier $\mathbb{Z}$-divisor such that $D-(K_X+\Delta)$ is $f$-ample and such that $(X, B+\Supp(E))$ is dlt. Then, $R^qf_*\sO_X(D)=0$ for all $q>0$ in the following three cases; \begin{enumerate}
\item If $\dim(f(E_i))\geq 1$ for all $i$.
\item If $\dim f(E_i)=0$ only for such $i$'s such that $E_i^{\nu}$ is a surface of log del Pezzo type.
\item There exists a $B'\geq B$ with standard coefficients, such that $K_X+B' +E\sim_{f,\mathbb{Q}} 0$, $(X,B'+E)$ is dlt, and $p>p_0$ as in Lemma \autoref{point}. 
\end{enumerate}

\end{prop}

\begin{proof} By the assumptions on the residue characteristic, \cite[Corollary 1.3]{ABL} and \cite[Theorem 3.1]{BK} implies that every rank one reflexive sheaf on $X$ is $S_3$. By Proposition \autoref{birKVV} it is therefore sufficient to show that the induced morphisms, $g_i\colon E^{\nu}_i\to C^{\nu}$, satisfies KVV for every component $E_i$ of $\Exc(f)$. However, $g_i\colon E^{\nu}_i\to C^{\nu}$ satisfies KVV under the first assumption by Lemma \autoref{curve} and under the second assumption by Lemma \autoref{curve} and \cite[Theorem 1.1]{ABL}. Finally Lemma \autoref{point} and Lemma \autoref{curve}  implies the theorem in the third case.
\end{proof}

\begin{prop}[KVV for dlt--modifications I]\label{dltmodif} Let $(X,\Delta)$ be a quasi--projective 
log canonical threefold without zero-dimensional log canonical centers and with closed points of perfect residue field of characteristic $p>5$. 
There exists a dlt-modification  $f\colon (X^{dlt}, \Delta_{X^{dlt}}+E)\to  (X,\Delta)$ 
of $(X,\Delta)$ such that $E=\Exc(f)$ is an integral divisor with $f^*(K_X+\Delta)= K_X + \Delta_{X^{dlt}}+E$ and such that $f$-satisfies the birational Kawamata--Viehweg vanishing (Proposition \autoref{KVVmain}). In particular, 
\begin{itemize} \item  (Strong) GR-vanishing holds for $f$ i.e., $R^if_*\sO_X(D)=0$ for any $\mathbb{Q}$-Cartier divisor $D$ such that $D-(K_{X^{dlt}}-\Delta')$ is f-nef for some $\mathbb{Q}$-divisor $\Delta'$, s.t., $0\leq \Delta'\leq \Delta_{X^{dlt}}+E$ and $\lfloor \Exc(\Delta')\rfloor=0$.
\item $R^if_*\sO_X(-E)=0$ for $i=1,2$ and $R^if_*\sO_X\cong R^if_{|E}\sO_E$ for $i>0$.
\end{itemize}
 \end{prop}

\begin{proof}By Proposition \autoref{dlt with ample} there exists a crepant dlt--model $f\colon (X^{dlt}, \Delta_{X^{dlt}}+E)\to  (X,\Delta)$ that contains an $f$-ample divisor $A$ supported on $\Exc(f)=E$, such that each component of $E$ is $\mathbb{Q}$-Cartier and $f^*(K_X+\Delta)= K_{X^{dlt}}+\Delta_{X^{dlt}}+E$. Since $(X,\Delta)$ contains no zero-dimensional lc-centers we have that $\dim(f(E_i))\geq 1$ for every component $E_i$ of $\Exc(f)$ and hence all the conditions of Proposition \autoref{KVVmain} are met.\\ We now show that this implies the strong GR-vanishing for $f$. To this end, let $-\epsilon A=\sum_{i} a_iE_i$, for some small enough $a_i>0$, be $f$-anti-ample.  Write $D\equiv L +K_{X^{dlt}}+\Delta'$ where $L$ is a nef $\mathbb{R}$-divisor and $0\leq \Delta'\leq \Delta_{X^{dlt}}+E$ is a $\mathbb{Q}$-divisor with $\lfloor \Delta'\cap E\rfloor=0$. Then $D-(K_{X^{dlt}}+\Delta'+\sum_ia_iE_i)\equiv \epsilon A+L $ which is $f$-ample. Moreover, since $\lfloor \Delta'\cap E\rfloor=0$ the pair $(X,\Delta'+\sum_ia_iE_i)$ is dlt. Proposition \autoref{KVVmain} implies that $R^if_*\sO_X(D)=0$ for $i>0$. \\
We now prove that $R^if_*\sO_X(-E)=0$, for $i=1,2$, and that $R^if_*\sO_X\cong R^if_{|E}\sO_E$ for $i>0$. By the long exact sequence of cohomology applied to the standard short exact sequence of $E$ inside $X$, the isomorphism  $R^if_*\sO_X\cong R^if_{|E}\sO_E$, for $i>0$, follows directly from the claimed vanishing $R^if_*\sO_X(-E)=0$, for $i=1,2$. The latter is just a special case of strong GR vanishing. Indeed, $-E-K_{X^{dlt}}-\Delta_{X^{dlt}}\sim_{f,\mathbb{Q}}0$, and is hence $f$- nef and $\lfloor \Delta_{X^{dlt}}\cap E\rfloor=0$.
\end{proof}

\begin{prop}[KVV for dlt--modifications II]\label{dltmodifII} There exists an integer $p_0$ with the following property. Assume that $\Delta\geq 0$ is a boundary with standard coefficients. Let $(X,\Delta)$ be a quasi--projective log canonical threefold  with closed points of perfect residue fields, each of characteristic $p>p_0$.  Let $f\colon  (X^{dlt}, \Delta_{X^{dlt}}+E)\to (X,\Delta)$ be the dlt modification in Proposition \autoref{dlt with ample}. Then,
\begin{itemize} \item  (Strong) GR-vanishing holds for $f$ i.e., $R^if_*\sO_X(D)=0$ for any $\mathbb{Q}$-Cartier $\mathbb{Z}$-divisor $D$ such that $D-(K_{X^{dlt}}+\Delta')$ is f-nef for some $\mathbb{Q}$-divisor $0\leq \Delta'\leq \Delta_{X^{dlt}}+E$ and $\lfloor \Exc(\Delta')\rfloor=0$.
\item $R^if_*\sO_X(-E)=0$ for $i=1,2$ and $R^if_*\sO_X\cong R^if_{|E}\sO_E$ for $i>0$.
\end{itemize}
 \end{prop}

\begin{proof} Let $p_0>5$ be as in  Lemma \autoref{point}. Since $(X^{dlt}, \Delta_{X^{dlt}}+E)$ has standard coefficients the third condition of Proposition \autoref{KVVmain} is satisfied. Strong GR-vanishing hence follows from Proposition \autoref{KVVmain} (3) and remark \autoref{assumfnef}. The vanishing $R^if_*\sO_X(-E)=0$, for $i=1,2$, follows from strong GR-vanishing as in the proof of Proposition \autoref{dltmodif}.
\end{proof}

\section{Main Theorems}

\begin{cor}\label{GR}Let $(X,\Delta)$ be a quasi--projective log canonical threefold without zero-dimensional log canonical centers. Assume that the closed points of $X$ have perfect residue fields of characteristic $p>5$. If $f\colon (Y, \Delta_Y)\to (X, \Delta)$ is crepant with $ (Y, \Delta_Y)$ log smooth then (weak) GR-vanishing holds for $f$ i.e., $R^if_*\sO_X(D)=0$ for any $\mathbb{Q}$-Cartier $\mathbb{Z}$-divisor $D$ such that $D\sim_{\mathbb{Q},f}K_Y-\Delta'$  for some effective $\mathbb{Q}$-divisor $\Delta'$, with $f_*\Delta'\leq \Delta$, and $\lfloor \Exc( \Delta') \rfloor=0$.
\end{cor}

\begin{proof} By \cite[Theorem 2.12]{BK} it is sufficient to prove the statement for some log resolution. Starting from a log resolution $ (Y,\Delta_Y)$ as in the proof of Proposition \autoref{dlt with ample} we get a rational morphism $\psi\colon (Y,\Delta_Y)\dashedrightarrow (X^{dlt}, \Delta^{dlt}+E)$ over $X$ where $h \colon (X^{dlt}, \Delta^{dlt}+E) \to (X, \Delta)$ is a dlt model containing an exceptional ample divisor. Upon replacing $ (Y, \Delta_Y)$ with a log resolution of $(X^{dlt}, \Delta^{dlt}+E)$ it is sufficient to prove that GR-vanishing holds for $f$ admitting a crepant factorization as in the diagram below:

$$\begin{tikzcd} (Y,\Delta_Y)\arrow{rd}[swap]{f} \arrow{r}{g}&  (X^{dlt}, \Delta^{dlt}+E)\arrow{d}{h}\\
&(X, \Delta)
\end{tikzcd}.$$

Let $D\sim_{f,\mathbb{Q}} K_Y+\Delta'$ for $f_*\Delta'\leq \Delta$, and let $D_{X^{dlt}}:=g_*D$ then \cite[Claim 2.5, Theorem 2.12]{BK} implies that $g_*\sO_Y(D)=\sO_{X^{dlt}}(D_{X^{dlt}})$ and all higher direct images $R^ig_*\sO_X(D)=0$, for all $i>0$. Let $\Delta'_{X^{dlt}}:=g_*\Delta'$,  then $D_{X^{dlt}}\sim_{\mathbb{Q},h}K_{X^{dlt}}+\Delta'_{X^{dlt}}$, where $h_*\Delta'_{X^{dlt}}=f_*\Delta'$ and $\lfloor \Delta'_{X^{dlt}} \cap \Exc(h)\rfloor=0$. 
By Proposition \autoref{dltmodif} $R^ih_*\sO_X(D_{X^{dlt}})=0$ for all $i>0$. We therefore conclude, by the composition of right derived functors, that $R^if_*\sO_X(D)=0$ for all $i>0$. 

\end{proof} 

\begin{cor}\label{GRII} There exists an integer $p_0>0$ with the following property. Let $(X,\Delta)$ be a quasi--projective log canonical threefold. Assume that $\Delta$ has standard coefficients. Assume that the closed points of $X$ have perfect residue fields of characteristic $p>p_0$. Then the weak GR-vanishing theorem holds over $(X,\Delta)$. I.e., if $f\colon (Y, \Delta_Y)\to (X, \Delta)$ is crepant with $ (Y, \Delta_Y)$ log smooth then $R^if_*\sO_X(D)=0$ for any $\mathbb{Q}$-Cartier $\mathbb{Z}$-divisor $D$ such that $D\sim_{\mathbb{Q},f}K_Y-\Delta'$  for some effective $\mathbb{Q}$-divisor $\Delta'$, with $f_*\Delta'\leq \Delta$ and $\lfloor \Exc(\Delta') \rfloor=0$.
\end{cor}

\begin{proof} By replacing the reference to Proposition \autoref{dltmodif} with Proposition \autoref{dltmodifII}, in the proof of Corollary \autoref{GR}, the same proof works verbatim. 
\end{proof}

We can now recover a simplified version of the main technical result (Theorem 1.5) of \cite{ABP}. 
\begin{cor}\label{wild}Let $x\in (X,\Delta)$ be a log--canonical threefold singularity which is not a log--canonical center. Suppose that the closed points of $X$ have perfect residue field of characteristic $p>5$. Let $f\colon Y\to X$ be a dlt modification as in Proposition \autoref{dltmodif}, such that every component of $E:=\Exc(f)$ has coefficient one in the crepant boundary and maps surjectively onto a curve $C$. Then $$H^2_x(X,\mathcal{O}_X)\cong H^0_x(C,R^1{f_{|E}}_*\mathcal{O}_E).$$ \end{cor}

\begin{proof} There is an associated spectral sequence $$H_x^i(X,R^jf_*\mathcal{O}_Y)\Rightarrow H_{f^{-1}x}^{i+j}(Y,\mathcal{O}_Y)$$ which induces a five term exact sequence containing the terms:
$$\adjustbox{scale=0.85,center}{
\begin{tikzcd} H^{1}_{f^{-1}x}(Y, \mathcal{O}_Y) \arrow{r} & H^0_x(X, R^1\pi_*\mathcal{O}_Y)\arrow{r} &H_x^2(X, \mathcal{O}_{X}) \arrow{r} &H^{2}_{f^{-1}x}(Y, \mathcal{O}_Y)\end{tikzcd}}.$$
By \cite[Cor 1.3.]{ABL} and \cite[Theorem 3.1]{BK} $Y$ is Cohen--Macaulay.
Local duality therefore gives, for $k\geq 1$, that $H^{k}_{\pi^{-1}x}(Y, \mathcal{O}_Y)\cong (R^{3-k}\pi_*\omega_Y)_x$. In particular, $H^{1}_{\pi^{-1}x}(Y, \mathcal{O}_Y)=0$ for dimension reasons and $H^{2}_{\pi^{-1}x}(Y, \mathcal{O}_Y)\cong (R^1f_*\omega_Y)_x$ which is zero by Proposition \autoref{dltmodif}. Indeed, let $0<a_i<1$ be such that $A:=-\sum_ia_iE_i$ is $f$-ample, then $K_Y-K_Y+ A$ is f-ample and the boundary $-A$ satisfies the assumptions of Proposition \autoref{dltmodif}. We therefore find $$H^0_x(X, R^1f_*\mathcal{O}_Y)\cong H_x^2(X, \mathcal{O}_{X}).$$
Lastly, Proposition \autoref{dltmodif} tells us that $R^1f_*\mathcal{O}_Y\cong R^1{f_{|E}}_*\mathcal{O}_E$. This proves the assertion.
\end{proof} 

\begin{remark}\label{wildrem} Corollary \autoref{wild} says that Koll\'ar's Theorem \autoref{lkvv} fails to be true at a non lc-center $x$ of $(X,\Delta)$ \emph{precisely} when the exceptional divisor of a dlt modification of $(X, \Delta)$ has a \emph{wild fiber} over $x$, see \cite[Section 4]{ABP} for a discussion on wild fibers. 
\end{remark}

\begin{Example}\label{Ex} Koll\'ar's Theorem \autoref{lkvv} can fail in every positive characteristic. 
Let $f\colon S\to C$ be a minimal wild elliptic surface in characteristic $p>0$. I.e., $f$ is of relative Picard rank one, $S$ and $C$ are smooth, $K_S\sim_{f,\mathbb{Q}} 0$ and $R^1f_*\mathcal{O}_S$ has torsion supported at some $x\in C$, i.e., $H^0_x(C,R^1f_*\mathcal{O}_S)\neq 0$. See \cite{kaue85} for examples of such elliptic fibrations. Let $\mathcal{L}$ be a relatively ample line bundle and consider the relative cone $\Spec_C(\bigoplus_i f _*\mathcal{L}^i)$. This is a log canonical threefold singularity with unique minimal log-canonical center the relative vertex, which is isomorphic to $C$. The point $x\in C$ is a non log canonical center, for which Koll\'ar's Theorem fails. Indeed, a dlt-modification as in Corollary \autoref{wild} is given by 
 $$\pi\colon \Spec_S(\bigoplus_i\mathcal{L}^i)\to \Spec_C(\bigoplus_i f _*\mathcal{L}^i),$$ with exceptional divisor of $\pi$ isomorphic to $S$ and $\pi_{|_S}=f\colon S\to C$. Therefore, Theorem \autoref{wild} implies that $H^2_x(X,\mathcal{O}_X)\cong H^0_x(C,R^1f_*\mathcal{O}_S)\neq 0$. This example is described in more detail in \cite[Section 5]{ABP}.
\end{Example}

We now mention applications of these vanishing theorems to the study of log-canonical centers (j.w., Quentin Posva) and to the study of the depth of log canonical singularities at a non-log canonical center (j.w., F. Bernasconi and Zs. Patakfalvi). 

\begin{thm}\label{AP}\cite[Theorem 1]{AP} Let $(X,\Delta)$ be a log canonical threefold whose closed points have perfect residue field of characteristic $p>5$. Then all minimal log canonical centers of $(X,\Delta)$ are normal. 
\end{thm} 

\begin{thm}\cite[Theorem 2]{AP}\label{AP2} Let $p_0$ be as in Proposition \autoref{dltmodifII}. Let $(X,\Delta)$ be a log canonical threefold whose closed points have perfect residue field of characteristic $p>p_0$. Then the union of all of the log canonical centers of $(X,\Delta)$ is semi-normal. 
\end{thm}

Finally, using Proposition \autoref{dltmodif}  instead of  \cite[Theorem 1.5]{ABP} 
 one can shorten the proof of the following result, due to \cite[Theorem 1.4]{ABP}.

\begin{cor}\label{S2} Let $(X,S+\Delta)$ be a log canonical threefold singularity with closed points of perfect residue field of characteristic $p>5$, where $S$ is a reduced effective Cartier divisor. Then $S$ is $S_2$. 
\end{cor}
\begin{proof} 

The following proof follows the lines of \cite[Theorem 1.4]{ABP}, except that we use Proposition \autoref{dltmodif} instead of \cite[Theorem 1.5]{ABP}.
Let $x\in S$ be a closed point, we want to show that $H^1_x(S,\mathcal{O}_S)=0$. By the long exact sequence of local cohomology associated to the standard short exact sequence of $S\subset X$, it is sufficient to prove that $H^1_x(X,\mO_X)=0$ and $H_x^2(X,\mO_X(-S))=0$. Since $S$ is Cartier, $X$ is $S_2$, and local cohomology can be computed locally around $x$ \cite[Prop. 1.3]{Grolc}, we only need to show $H^2_x(X,\mO_X)=0$. Since $x\in S>0$ we can not have that $x$ is a log canonical center of $(X,\Delta)$. Without loss of generality, we may assume that $x\in C$ where $C$ is a one dimensional minimal log canonical center (since otherwise  $H^2_x(X,\mO_X)=0$ by \cite{ABL, BK}). By Corollary \autoref{wild} there exists a dlt--modification  $f\colon (Y,\Delta_Y+E)\to (X,\Delta)$
with reduced exceptional divisor $E$ and such that  $H^2_x(X,\mathcal{O}_X)\cong H^0_x(C,R^1{f_{|E}}_*\mathcal{O}_E)$. 
By Theorem \autoref{AP}, the curve $C$ is normal and therefore ${f_{|E}}_*\mathcal{O}_E=\mathcal{O}_C$, and by miracle flatness \cite[Tag 00R4]{stacks-project}, $f_{|E}\colon E\to C$ is flat. By \cite[Theorem 3.1]{BK}, the divisor $E$ is $S_2$ and each component $E_i$ is normal. Assume in order to arrive at a contradiction that $H^0_x(C,R^1{f_{|E}}_*\mathcal{O}_E)\neq 0$. Then, \cite[Theorem 8.2.1]{Ray}  implies that each component of the fiber over $x\in C$ is multiple (see \cite[Section 4]{ABP} for more details on wild fibers). By construction $(Y, \Delta_Y+E+f^*(S))$ is log canonical, and hence by adjunction
$(E_i, \Diff_{E_i}(\Delta_Y+E-E_i+f^*(S))$ is log canonical for every $i$. Since $f^*(S)$ is Cartier, $f^*(S)_{|E_i}$ is a component of  $\Diff_{E_i}(\Delta_Y+E-E_i+f^*(S))$ for every $i$, \cite[p. 154]{Kolsing}. But $f^*(S)_{|E_i}$ contains an irreducible component of the fiber of $f_{|E}$ over $x$, which we have seen to be multiple. This contradicts that $(E_i, \Diff_{E_i}(\Delta_Y+E-E_i+f^*(S))$ is log canonical. 
\end{proof}

\bibliography{phdbib-1}{}
\bibliographystyle{alpha}

\end{document}